\documentclass[a4paper, 11 pt]{article}
\usepackage{a4wide, amsmath, amssymb, mathtools, yfonts}
\usepackage{comment}
\usepackage{tikz}
\usepackage{tikz-cd}
\usepackage[all]{xy}
\usepackage[utf8]{inputenc}
\usepackage{amsthm, mathrsfs}
\usepackage[english]{babel}
\usepackage{hyperref}
\usepackage{authblk}
\usepackage[OT2,T1]{fontenc}
\DeclareSymbolFont{cyrletters}{OT2}{wncyr}{m}{n}
\DeclareMathSymbol{\Sha}{\mathalpha}{cyrletters}{"58}

\numberwithin{equation}{section}
\newtheorem{lemma}{Lemma}[section]
\newtheorem{theorem}[lemma]{Theorem}
\newtheorem{proposition}[lemma]{Proposition}
\newtheorem{corollary}[lemma]{Corollary}

\newtheorem{conjecture}{Conjecture}

\theoremstyle{definition}

\newcommand{\Z}{\mathbb{Z}}
\newcommand{\Q}{\mathbb{Q}}

\newcommand{\R}{\mathbb{R}}

\newcommand{\CL}{\mathrm{Cl}}

\newcommand\Gal{\mathrm{Gal}}

\usepackage{csquotes}
\usepackage{IEEEtrantools}
\newcommand{\F}{\mathbb{F}}

\title{\vspace{-\baselineskip}\sffamily\bfseries Elliptic curves and spin}
\author[1]{Peter Koymans}
\author[2]{Peter Vang Uttenthal}
\affil[1]{ETH Zurich}
\affil[2]{Aarhus University}
\date{\today}

\begin{document}

\maketitle

\begin{quote}
It is interesting to recall that, in connexion with a lecture by Prof. R. Fueter at the 1932 Zurich Congress, Hilbert asserted that the theory of complex multiplication of elliptic modular functions which forms a powerful link between number theory and analysis, is not only the most beautiful part of mathematics but also of all science.
\flushright{Olga Taussky, \emph{David Hilbert For.Mem.R.S}, Nature, 1943 \cite{olga}}
\end{quote}

\begin{abstract}
In the early 2000s, Ramakrishna asked the question: For the elliptic curve 
$$
E: y^2 = x^3 - x,
$$
what is the density of primes $p$ for which the Fourier coefficient $a_p(E)$ is a cube modulo $p$? As a generalization of this question, Weston--Zaurova \cite{Weston-Zaurova} formulated conjectures concerning the distribution of power residues of degree $m$ of the Fourier coefficients of elliptic curves $E/\Q$ with complex multiplication. In this paper, we prove the conjecture of Weston--Zaurova for cubic residues using the analytic theory of spin. Our proof works for all elliptic curves $E$ with complex multiplication. 
\end{abstract}


\section{Introduction}
For the elliptic curve
$$
E: y^2 = x^3 - x,
$$
with complex multiplication by $\Q(i)$ and error terms
$$
a_p(E) = p+1 - |E(\mathbb{F}_p)|,
$$
are there infinitely many primes $p \equiv 1 \bmod 12$ such that $a_p(E)$ is a cubic residue modulo $p$? This question was formulated by Ramakrishna in the early 2000s. At that time, the question was perceived as interesting yet inaccessible. Note that it is natural to restrict to primes $p \equiv 1 \bmod 12$. Indeed, for $p \equiv 3 \bmod 4$, we have $a_p(E) = p + 1$ if $p \geq 5$, so $a_p(E)$ is a cube modulo $p$. Furthermore, all elements of $\mathbb{F}_p$ are automatically cubes if $p \equiv 2 \bmod 3$, so in particular $a_p(E)$ is once more a cube modulo $p$.

In 2005, Weston \cite{Weston} formulated a family of conjectures inspired by the question, but only for elliptic curves without complex multiplication. In a subsequent paper, Weston and Zaurova considered the analogous conjectures for elliptic curves over $\Q$ with complex multiplication. Write $w_K$ for the number of roots of unity contained in a number field $K$.

\begin{conjecture}[{Weston--Zaurova \cite[Conjecture 3.1]{Weston-Zaurova}}]
\label{WZ conjecture}
Let $E$ be an elliptic curve over $\Q$ with complex multiplication by an imaginary quadratic field $K$. Let $m$ be an integer coprime with $w_K$. Then the density of primes $p \equiv 1 \bmod m$ with $a_p(E) \neq 0$ for which $a_p(E)$ is an $m^{th}$ power residue modulo $p$ is $1/m$. Furthermore, the density is independent of any Chebotarev class $\mathcal{P}$ contained in the set of primes $p\equiv 1 \bmod m$ with $a_p(E)\neq 0$ in the sense that if the primes are restricted to $\mathcal{P}$, the relative density remains $1/m$.
\end{conjecture}

Note that, for all but finitely many primes, $p\equiv 1 \bmod m$ and the reduction of $E$ at $p$ being ordinary is equivalent to $p$ splitting completely in $K(\zeta_m)$. 

In \cite{Weston} and \cite{Weston-Zaurova}, Weston and Zaurova prove the special cases of the conjecture to which classical techniques are applicable. In particular, they settle the case when $K = \Q(\sqrt{-3})$ and $m=3$ by using that $K$ contains the third roots of unity so that the classical cubic reciprocity law is available. For this reason, we exclude the field $K = \Q(\sqrt{-3})$ in the statements of this paper. 

In contrast, the classical techniques do not apply to the case $K=\Q(i)$ and $m = 3$, and therefore the original question for the curve $E: y^2 = x^3 - x$ remained open. In this paper, we answer Ramakrishna's question in the affirmative, and prove that among the primes $p \equiv 1 \bmod 3$ for which $E: y^2 = x^3 - x$ has ordinary reduction (equivalently, $p \equiv 1 \bmod 4$ for $p \geq 5$), the density of $a_p(E)$ being a cubic residue modulo $p$ is $1/3$. In addition, we prove Conjecture \ref{WZ conjecture} for all elliptic curves with complex multiplication (CM) by an imaginary quadratic field $K = \Q(\sqrt{-d})$ where $d \neq 3$ in the case of cubic power residues (i.e. $m = 3$). We do not need the curves to be defined over $\Q$, nor do we need them to have CM by a maximal order; it suffices that they have CM by some order $\mathcal{O}$ in $K$ and are defined over the ring class field $L$ of $\mathcal{O}$.

To prove the main theorems of this paper, we extend the analytic theory of spin to cubic symbols over totally imaginary biquadratic fields. Our results are \emph{unconditional}, as we do not need to assume any standard conjectures on short character sums. This is atypical for results on spin symbols, which are most often conditional on such conjectures. 

\begin{theorem} 
\label{sieve}
Let $K$ be an imaginary quadratic field satisfying $\gcd(3, w_K) = 1$, or equivalently $K \neq \Q(\zeta_3)$. For any prime $p$ that splits completely in $K$ as $(\pi \overline{\pi})$, where $\overline{\pi}$ is the conjugate of $\pi$, and where $\pi$ lies below a prime ideal $\mathfrak{p}$ in $K(\zeta_3)$, there is a well-defined cubic spin symbol of the form 
$$
[\mathfrak{p}] = \left(\frac{\overline{\pi}}{\mathfrak{p}}\right)_{K(\zeta_3),3}.
$$
Furthermore, there is a constant $C > 0$ such that for all $X \geqslant 100$
$$
\left|\sum_{ N_{K(\zeta_3)/\Q}(\mathfrak{p})\leqslant X} [\mathfrak{p}]\right| \leqslant C X^{1 - \frac{1}{3200}}.
$$
\end{theorem}

One readily checks that the symbol does not depend on the choice of generator $\pi$, see Proposition \ref{spin and a_p}. The last part of Theorem \ref{sieve} lies much deeper and its proof occupies the majority of this paper.

As a corollary, for any elliptic curve $E$ with complex multiplication by an imaginary quadratic field $K$, we prove asymptotics with a power saving error term for the primes $p$ that split completely in $K(\zeta_3)$ such that $a_p(E)$ is a cubic residue modulo $p$. Let us now explain what we mean by $a_p(E)$ in this context.

Let $E$ be an elliptic curve with CM by an order $\mathcal{O}$ in an imaginary quadratic field $K$ and let $L$ be the ring class field of $\mathcal{O}$. Assume that $E$ is defined over $L$. If we take a prime $\mathfrak{p}$ above $p$ where $E$ has good reduction, we may introduce the quantity
\begin{align}
\label{eapEgeneral}
a_\mathfrak{p}(E) = p + 1 - |E(O_L/\mathfrak{p})|.
\end{align}
If $\mathcal{O}^\times = \{\pm 1\}$ and if $\mathfrak{q}$ is another prime of $L$ above $p$, it is a classical result of Deuring that $a_{\mathfrak{p}}(E)^2 = a_{\mathfrak{q}}(E)^2$, i.e. $a_\mathfrak{p}(E)$ is determined up to sign (for a reference, see our Lemma \ref{magic}). Therefore $a_\mathfrak{p}(E)$ being a cube modulo $p$ does not depend on the choice of prime $\mathfrak{p}$. A somewhat more involved argument gives the same result if $\mathcal{O}^\times = \{\pm 1, \pm i\}$, see Proposition \ref{spin and a_p}.

Since it is possible to impose arbitrary congruence conditions when handling spin symbols, we are able to get density results even when imposing arbitrary \emph{abelian splitting conditions}.

\begin{corollary} 
\label{ellipticcurvedensity}
Let $E$ be an elliptic curve with complex multiplication by an order $\mathcal{O}$ in an imaginary quadratic field $K$ different from $\Q(\zeta_3)$. Suppose $E$ is defined over the ring class field $L$ of $\mathcal{O}$. Then 
$$
\frac{\# \{p\leqslant X, p\equiv 1 \bmod 3,  p \textup{ splits in } L: a_{\mathfrak{p}}(E) \textup{ is a cube modulo } p\}}{\#\{p\leqslant X, p\equiv 1 \bmod 3, p \textup{ splits in } L\}}
=
\frac{1}{3} + O\left(\frac{1}{X^{1/3200}}\right).
$$
Here $a_{\mathfrak{p}}(E)$ is defined as in equation \eqref{eapEgeneral} with $\mathfrak{p}$ an arbitrary prime of $L$ above $p$.

If $\mathcal{P}$ is a Chebotarev class of rational primes defined by a splitting condition in an abelian extension of $K(\zeta_3)$ containing $L$ and projecting trivially in $\Gal(L/\Q)$, then the result remains valid when the set of primes is restricted to $p$ with $p \in \mathcal{P}$.
\end{corollary}


For an imaginary quadratic field $K$ and a prime $p$, we let $K^{(m)}(p)$ be the maximal abelian extension of $K$ killed by $m$ and unramified away from $p$. We prove the following density theorem on the splitting of primes $p$ in the number fields $K^{(m)}(p)$ that depend on $p$. In the statement of Theorem \ref{K^p}, the splitting condition on $p$ is imposed to control the size of the maximal abelian extension of $K$ unramified away from $p$. 

Let $m$ be a prime and let $\zeta_m$ be a primitive $m^{\operatorname{th}}$ root of unity. Let $F$ be a number field, and define 
$$
V_\emptyset = \{ x\in F^\times : (x) = J^m \text{ for some fractional ideal }J \},
$$
where $(x)$ is the fractional ideal generated by $x\in F^\times$. We define the $m$-governing field of $F$ to be the field $F(\zeta_m, \sqrt[m]{V_\emptyset}).$

\begin{theorem}[Splitting in number fields that depend on the prime] 
\label{K^p} 
Let $K$ be an imaginary quadratic field and let $m \geq 3$ be a prime such that $\gcd(m, w_K) = 1$. Let $\mathcal{O}$ be an order in $K$, and let $p$ be a prime that splits completely in the ring class field of $\mathcal{O}$ and the $m$-governing field of $K$. Let $K^{(m)}(p)$ be the maximal $m$-elementary abelian extension of $K$ unramified away from $p$, and let $f(p, K^{(m)}(p)/\Q)$ be the inertial degree at $p$ in $K^{(m)}(p)/\Q$. Then $[\mathfrak{p}] = 1$ if and only if $f(p, K^{(m)}(p)/\Q) = 1$. 

In particular, for $m = 3$, among the primes $p$ splitting completely in the ring class field of $\mathcal{O}$ and the $3$-governing field, the set of primes $p$ of residue field degree $1$ at any prime in $K^{(3)}(p)$ above $p$ has density $1/3$. 
\end{theorem}

In this theorem, we are using implicitly that the density results on spin symbols remain valid when the class of primes considered is restricted to Chebotarev classes defined by Chebotarev conditions in abelian extensions of $K(\zeta_3)$. We expect that it is possible to prove the Weston--Zaurova conjecture for all power residues $m$ by applying techniques from the joint distribution of spins from Koymans--Milovic \cite{KM2}. However, such results would be conditional on short character sum conjectures. 

Below, we state a corollary of Theorem \ref{sieve}. 

\begin{corollary} 
\label{galois}
Let $\ell$ be a prime, and let 
$$
\rho: \Gal(\overline{\Q}/\Q) \longrightarrow \operatorname{GL}(2, \Z_\ell)
$$
be a Galois representation arising from the Galois action on $\ell^\infty$-torsion points of an elliptic curve over $\Q$ with complex multiplication by an imaginary quadratic field $\Q(\sqrt{-d})$ where $d \neq -3$. Let $S$ be the smallest set of places containing $\ell$ and the archimedean place of $\Q$ such that $\rho$ is unramified outside $S$. For primes $p \not \in S$, let $\operatorname{Frob}_p$ be any lift of Frobenius to $\Gal(\overline{\Q}_p/\Q_p)$. Among the primes $p\equiv 1 \bmod 3$ splitting completely in the ring class field of $\operatorname{End}_{\mathbb{C}}(E)$,
$$
\operatorname{trace}(\rho(\operatorname{Frob}_p)) \in \Z
$$
is a cube modulo $p$ one third of the time. 
\end{corollary}

Our method of proof of Theorem \ref{sieve} relies on a version of Vinogradov's sieve extended to totally imaginary biquadratic fields. In \cite{FIMR}, Friedlander--Iwaniec--Mazur--Rubin introduced quadratic spin symbols of primes in totally real, cyclic number fields (satisfying some additional technical conditions) and proved equidistribution of such symbols by extending Vinogradov's sieve to totally real, cyclic fields. In general, their results are conditional on standard conjectures on short character sums (denoted $C_n$ in their paper where $n$ is the degree of the totally real cyclic number field), and their results are only unconditional in the case $n = 3$. In this paper, we extend the methods in \cite{FIMR} to \emph{cubic} spin symbols over biquadratic, totally imaginary fields of degree 4 over $\Q$. It is natural to compare this to work of Merikoski \cite{Meri}, although we should warn the reader that his spin symbol is not the same as ours.

Our main analytic accomplishment is to make our equidistribution results unconditional in this setting. To make our results unconditional, we follow a similar strategy strategy as employed previously by Koymans--Milovic \cite{KM1} and Piccolo \cite{Piccolo}. However, unlike the works \cite{KM1, Piccolo}, we do not assume that our biquadratic field has class number $1$ and we deal with cubic symbols instead of quartic symbols.

\subsection{Acknowledgements}
The authors are thankful to Ravi Ramakrishna and Fabien Pazuki for helpful conversations. The first author gratefully acknowledges the support of Dr. Max R\"ossler, the Walter Haefner Foundation and the ETH Z\"urich Foundation.

\section{General spin symbols}
In this section we introduce our cubic spin symbol of interest, and relate it to $a_p(E)$. This symbol is a natural generalization of the spin symbol introduced in \cite{FIMR}. In fact, it turns out that the algebraic part of our theory readily generalizes to all $m$, not just $m = 3$. We start by identifying a family of spin symbols, indexed by $m$, capturing whether the error terms $a_p(E)$ are power residues modulo $m$. This is achieved in Proposition \ref{spin and a_p}. We will then relate this to splitting of $p$ in a number field depending on $p$, see Lemma \ref{lSplit2}.

\subsection{\texorpdfstring{Power residues of degree $m$ and spin}{Power residues of degree m and spin}}
Let $K$ be an imaginary quadratic field with maximal order $\mathcal{O}_K$. For any order $\mathcal{O}$ of conductor $f$ in $K$, there is a unique abelian extension $L/K$ called the ring class field of $\mathcal{O}$. We refer to \cite[p. 179]{cox} for the definition of ring class fields.

We write $\sigma$ for the non-trivial element of $\Gal(K/\Q)$. To keep our notation brief, we shall frequently use conjugation to denote this automorphism, so that we may speak of $\bar{\kappa}$ for elements $\kappa$ but also of $\bar{\mathfrak{p}}$ for (prime) ideals $\mathfrak{p}$.



\begin{lemma} 
\label{magic}
Let $K$ be an imaginary quadratic field and let $\mathcal{O}$ be an order in $K$. 
Let $L$ be the ring class field of $\mathcal{O}$, and let $E$ be an elliptic curve over $L$ with $\operatorname{End}_\mathbb{C}(E) = \mathcal{O}$.
Let $p$ be a prime that splits completely in $L$ and $\mathfrak{p}$ be a prime in $L$ above $p$
(in particular, $O_L/\mathfrak{p} \simeq \F_p$).
Suppose $E$ has good reduction at $\mathfrak{p}$. 
Then there is $\kappa \in \mathcal{O}$ such that $p = \kappa \overline{\kappa}$ and 
$$
a_\mathfrak{p}(E) =  p + 1 - |E(O_L/\mathfrak{p})| = \kappa + \overline{\kappa}.
$$
\end{lemma}

\begin{proof}
See \cite[Theorem 14.16]{cox}.
\end{proof}

\begin{proposition} 
\label{spin and a_p}
Let $K = \Q(\sqrt{D})$ be an imaginary quadratic field and let $m \geq 3$ be any odd integer such that $\gcd(m, w_K) = 1$. Let $M = \Q(\sqrt{D}, \zeta_m)$ and let 
$$
\Gal(M/\Q) \simeq \Gal(K/\Q) \times \Gal(\Q(\zeta_m)/\Q) 
= \{1, \sigma \} \times (\Z/m\Z)^\times.
$$
Let $\mathcal{O}$ be an order in $K$ and let $E$ be an elliptic curve defined over the ring class field of $\mathcal{O}$ with $\operatorname{End}_\mathbb{C}(E) = \mathcal{O}$. Let $p \equiv 1 \bmod m$ be a prime that splits completely in the ring class field of $\mathcal{O}$. According to Lemma \ref{magic}, there exists $\kappa \in \mathcal{O}$ such that 
$$
p = \kappa \overline{\kappa}
$$
and 
$$
a_\mathfrak{P}(E) = \kappa + \overline{\kappa},
$$
where $\mathfrak{P}$ is a prime of $L$ above $p$.

Let $\mathfrak{q} = (\kappa)$ and let $\mathfrak{p} $ be a prime in $M$ above $\mathfrak{q}$.
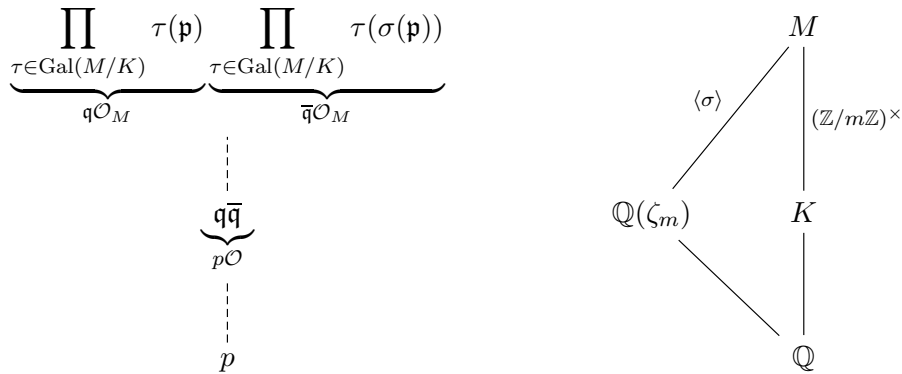
\[
\begin{tikzcd}
\underbrace{ \prod_{\tau \in \Gal(M/K)} \tau (\mathfrak{p}) }_{\mathfrak{q} \mathcal{O}_{M} } 
\underbrace{ \prod_{\tau \in \Gal(M/K)} \tau (\sigma(\mathfrak{p}))}_{\overline{\mathfrak{q}} \mathcal{O}_M } & &  & M & \\
\underbrace{\mathfrak{q} \overline{\mathfrak{q}} }_{p \mathcal{O}} \arrow[u,  dashed, dash] &   & 
\Q(\zeta_m) \arrow[ur, dash, "\langle \sigma \rangle"] & K \arrow[u, dash, swap, "(\Z/m\Z)^\times"] & \\
p \arrow[u, dashed, dash] & & & \Q \arrow[ul, dash] \arrow[u, dash] & 
\end{tikzcd}
\] 
Define the spin symbol 
$$
[\mathfrak{p}]
= \left( \frac{ \overline{\kappa} }{\mathfrak{p}} \right)_{M,m} = \left( \frac{ \sigma(\kappa) }{\mathfrak{p}} \right)_{M,m}.
$$
If $u \in \mathcal{O}^\times,$ then 
$$
\left(\frac{ \overline{u \kappa}}{\mathfrak{p}}\right)_{M,m} =  \left(\frac{ \overline{ \kappa} }{\mathfrak{p}}\right)_{M,m},
$$
so the spin symbol is well-defined, i.e. is independent on the choice of generator $\kappa$ for $\mathfrak{q}$. Furthermore, $[\mathfrak{p}] = 1$ if and only if $a_\mathfrak{P}(E)$ is an $m^{\operatorname{th}}$ power residue modulo $p$.
\end{proposition}

\begin{proof}
Since $p\equiv 1 \bmod m$, $p$ splits completely in $\Q(\zeta_m)$. Since we assume $p$ splits completely in $K,$ we conclude that $p$ splits completely in $M$. Since $\mathfrak{p}$ lies above $\mathfrak{q},$ $\sigma(\mathfrak{p})$ lies above $\sigma(\mathfrak{q}) = \overline{\mathfrak{q}}$. In $M$, the primes $\mathfrak{q}$ and $\overline{\mathfrak{q}}$ decompose as follows:
$$
\mathfrak{q} \mathcal{O}_{M} =  \prod_{\tau \in \Gal(M/K)} \tau (\mathfrak{p}),
\quad 
\overline{\mathfrak{q}} \mathcal{O}_{M} = \sigma(\mathfrak{q}) \mathcal{O}_{M}
=   \prod_{\tau \in \Gal(M/K)} \tau (\sigma(\mathfrak{p})).
$$
Since $\gcd(m, w_K) = 1$ by assumption, we see that all elements of $\mathcal{O}^\times$ are $m^{\operatorname{th}}$ powers in $K$. We conclude that the spin symbol $[\mathfrak{p}]$ is independent of the choice of $\kappa$. Using that $\kappa \equiv 0 \bmod \mathfrak{p}$, we have by Lemma \ref{magic}
$$
[\mathfrak{p}] = 
\left( 
\frac{\overline{\kappa}}{\mathfrak{p}}
\right)_{M,m} 
=
\left(\frac{\overline{\kappa} + \kappa}{\mathfrak{p}}\right)_{M, m}
=
\left(\frac{a_\mathfrak{P}(E)}{\mathfrak{p}}\right)_{M, m}.
$$
Hence $[\mathfrak{p}] = 1$ if and only if $a_\mathfrak{P}(E)$ is an $m^{\operatorname{th}}$ power residue modulo $\mathfrak{p}$. Now, $\mathcal{O}_M/\mathfrak{p} \simeq \F_p,$ so $a_\mathfrak{P}(E)$ is an $m^{\operatorname{th}}$ power modulo $\mathfrak{p}$ if and only if $a_\mathfrak{P}(E)$ is an $m^{\operatorname{th}}$ power modulo $p$.
\end{proof}

\subsection{Spin and inertial degrees of ramified primes}
Let $m$ be a prime. For any set $S$ of tame places in a number field $F$, define
$$
V_S = \{ x\in F^\times : (x) = J^m\text{ for some fractional ideal } J, \text{ and } x\in F_v^{\times m} \text{ } \forall v\in S\}. 
$$
Note that, by definition, $F^{\times m} \subseteq V_S$. Let $\mathcal{O}_F^\times$ be the units in $F$ and $\operatorname{Cl}_F[m]$ the $m$-torsion in the class group of $F$. There is an exact sequence
$$
1 \longrightarrow \mathcal{O}_F^\times \otimes \F_m \longrightarrow V_{\emptyset}/F^{\times m} \longrightarrow \operatorname{Cl}_F[m] \longrightarrow 1.
$$
The $m$-governing field of $F$ is the field $F(\zeta_m, \sqrt[m]{V_\emptyset})$, where $\zeta_m$ is a primitive $m^{\operatorname{th}}$ root of unity. 

Recall that $K^{(m)}(p)$ denotes the maximal $m$-elementary abelian extension of $K$ unramified away from the primes above $p$. Write $H_K^{(m)}$ for the maximal unramified $m$-elementary abelian extension of $K$.

\begin{lemma} 
\label{tamegrasmunnier}
Let $K$ be an imaginary quadratic field and let $m$ be a prime. Let $p$ be a prime that splits completely in the $m$-governing field of $K$, and factor $pO_K = \mathfrak{q} \overline{\mathfrak{q}}$. Then there exists a $\Z/m\Z$-extension $K^{(m)}(\mathfrak{q})/K$ ramified exactly at $\mathfrak{q}$ and a $\Z/m\Z$-extension $K^{(m)}(\overline{\mathfrak{q}})/K$ ramified exactly at $\overline{\mathfrak{q}}$. For any choice of $K^{(m)}(\mathfrak{q})/K$ and $K^{(m)}(\overline{\mathfrak{q}})/K$ as above, we have
$$
K^{(m)}(p) = K^{(m)}(\mathfrak{q}) K^{(m)}(\overline{\mathfrak{q}}) H_K^{(m)},
$$ 
and thus $[K^{(m)}(p) : \Q] = 2m^2 \cdot |\mathrm{Cl}(K)[m]|$. Diagramatically
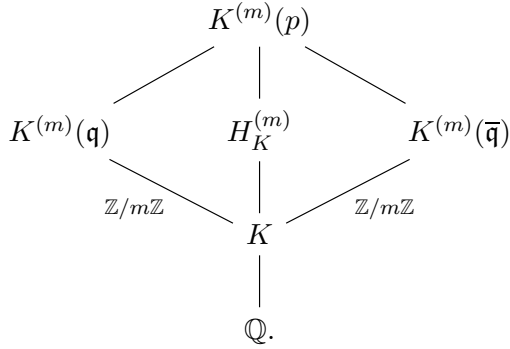
\[ 
\begin{tikzcd}
& K^{(m)}(p) & \\
K^{(m)}(\mathfrak{q}) \arrow[ur, dash] & \arrow[d, dash] H_K^{(m)} \arrow[u, dash] & K^{(m)}(\overline{\mathfrak{q}}) \arrow[ul, dash]\\
   & K \arrow[ul, dash, "\Z/m\Z"] \arrow[ur, dash, swap, "\Z/m\Z"] & \\
   & \Q. \arrow[u, dash] & \\ 
\end{tikzcd}
\]
\end{lemma}

\begin{proof}
Since $p$ splits completely in the $m$-governing field and since the $m$-governing field is ramified at $m$, it follows that $p \neq m$. Therefore, all primes of $K$ above $p$ must be tamely ramified in $K^{(m)}(p)/K$. 

By the tame Gras--Munnier theorem, there exists a $\Z/m\Z$-extension of $K$ ramified exactly at $\mathfrak{q}$. Likewise, there exists a $\Z/m\Z$-extension of $K$ ramified exactly at $\overline{\mathfrak{q}}$. The compositum $K^{(m)}(\mathfrak{q}) K^{(m)}(\overline{\mathfrak{q}}) H_K^{(m)}$ is an $m$-elementary abelian extension of $K$ unramified away from $p$ with degree $2m^2 \cdot |\mathrm{Cl}(K)[m]|$. Evidently, $K^{(m)}(\mathfrak{q}) K^{(m)}(\overline{\mathfrak{q}}) H_K^{(m)}$ is the maximal such extension and hence $K^{(m)}(p) = K^{(m)}(\mathfrak{q}) K^{(m)}(\overline{\mathfrak{q}}) H_K^{(m)}$.
\end{proof}

\begin{lemma}
Let $K$ be an imaginary quadratic field and let $m$ be a prime. Let $p$ be a prime that splits completely in the $m$-governing field of $K$. There are two possibilities for the standard $(e, f, g)$-decomposition of $p$ in $K^{(m)}(p)$:
\begin{table}[h!] 
\begin{center} 
\begin{tabular}{llll}  
$e(p, K^{(m)}(p)/\Q)$ & $f(p, K^{(m)}(p)/\Q)$ & $g(p, K^{(m)}(p)/\Q)$  & $[K^{(m)}(p) : \Q]$ \\ [1ex]\hline 
$m$ & $1$ & $2m \cdot |\mathrm{Cl}(K)[m]|$ & $2m^2 \cdot |\mathrm{Cl}(K)[m]|$ \\ 
$m$ & $m$ & $2 \cdot |\mathrm{Cl}(K)[m]|$ & $2m^2 \cdot |\mathrm{Cl}(K)[m]|$ \\   [1ex]
\end{tabular}
\caption{The $(e,f,g)$-decomposition of $p$ in $K^{(m)}(p)/\Q$}
\end{center}
\end{table}
\end{lemma}

\begin{proof}
Because $p$ is unramified in $K$ and tamely ramified in $K^{(m)}(p)/K$, and because tame ramification is cyclic, we have $e(p, K^{(m)}(p)/\Q) = m$. Since $f(p, K/\Q) = 1,$ and since $\Gal(K^{(m)}(p)/K)$ is a vector space over $\F_m$, it follows that $f(p, K^{(m)}(p)/\Q)$ is either 1 or $m$. By Lemma \ref{tamegrasmunnier}, we have
$$
e(p, K^{(m)}(p)/\Q) f(p, K^{(m)}(p)/\Q) g(p, K^{(m)}(p)/\Q) = 2m^2 \cdot |\mathrm{Cl}(K)[m]|,
$$
and the result follows.
\end{proof}

\begin{lemma} 
\label{lSplit1}
Let $m$ be a prime. Suppose $p$ splits completely in the $m$-governing field of a quadratic imaginary field $K$ and in $H_K^{(m)}$. Then we have for any choice of $K^{(m)}(\overline{\mathfrak{q}})$
$$
f(p, K^{(m)}(p)/\Q) = 1 \iff f(\mathfrak{q}, K^{(m)}(\overline{\mathfrak{q}})/K) = 1.
$$
\end{lemma}

\begin{proof}
Note that the compositum $K^{(m)}(\overline{\mathfrak{q}}) H_K^{(m)}$ does not depend on the choice of $K^{(m)}(\overline{\mathfrak{q}})/K$ as it is the maximal $m$-elementary abelian extension of $K$ unramified away from $\mathfrak{q}$. Since $p$ is assumed to split completely in $H_K^{(m)}$, we conclude that $f(\mathfrak{q}, K^{(m)}(\overline{\mathfrak{q}})/K)$ is independent of this choice.

Since $f(p, K/\Q) = 1$, we have $f(p, K^{(m)}(p)/\Q) = f(\mathfrak{q}, K^{(m)}(p)/K)$. Because $K^{(m)}(p) = K^{(m)}(\mathfrak{q}) K^{(m)}(\overline{\mathfrak{q}}) H_K^{(m)}$ and because $p$ splits completely in $H_K^{(m)}$, we conclude that
\begin{align}
\label{eEquivalence}
f(p, K^{(m)}(p)/\Q) = 1 &\iff f(\mathfrak{q}, K^{(m)}(p)/K) = 1 \nonumber \\
&\iff f(\mathfrak{q}, K^{(m)}(\mathfrak{q})/K) = 1 \text{ and }f(\mathfrak{q}, K^{(m)}(\overline{\mathfrak{q}})/K) = 1.
\end{align}
Since $e(\mathfrak{q}, K^{(m)}(\mathfrak{q})/K) = m$, the condition $f(\mathfrak{q}, K^{(m)}(\mathfrak{q})/K) = 1$ is automatically satisfied. Removing it from the equivalence \eqref{eEquivalence} gives the lemma.
\end{proof}

\begin{lemma} 
\label{lSplit2}
Let $K$ be an imaginary quadratic field and let $m \geq 3$ be any odd integer such that $\gcd(m, w_K) = 1$. Set $M := K(\zeta_m)$. Let $\mathcal{O}$ be an order in $K$. Let $p$ be a prime that splits completely in the ring class field of $\mathcal{O}$ and the $m$-governing field of $K$, and let $\kappa \in \mathcal{O}$ with $p = \kappa \overline{\kappa}$. Let $\mathfrak{q} = (\kappa)$, $\mathfrak{p}$ a prime in $M$ above $\mathfrak{q}$ and $[\mathfrak{p}] = (\overline{\kappa}/ \mathfrak{p})_{M, m}$. Then we have for any choice of $K^{(m)}(\overline{\mathfrak{q}})$
$$
f(\mathfrak{q}, K^{(m)}(\overline{\mathfrak{q}})/K) = 1 \iff [\mathfrak{p}] = 1. 
$$
\end{lemma}

\begin{proof}
Let $G(\overline{\mathfrak{q}})$ be the full ray class group of modulus $\overline{\mathfrak{q}}$ above $K$ and let $K(\overline{\mathfrak{q}})$ be the full ray class field of modulus $\overline{\mathfrak{q}}$ above $K$. If $\psi_{K(\overline{\mathfrak{q}})/K}$ is the Artin reciprocity map, there is a commutative diagram
\[
\begin{tikzcd}
G(\overline{\mathfrak{q}}) \arrow[r, "\psi_{K(\overline{\mathfrak{q}})/K}"] \arrow[d]
& \Gal(K(\overline{\mathfrak{q}})/K) \arrow[d] \\
G(\overline{\mathfrak{q}}) / G(\overline{\mathfrak{q}})^{m} \arrow[r] & \Gal(K^{(m)}(\overline{\mathfrak{q}}) H_K^{(m)}/K),
\end{tikzcd}
\]
where the horizontal arrows are isomorphisms. Let $\operatorname{Cl}_K$ be the class group of $K$. By Theorem 1.7, Chapter V, in \cite{milneCFT}, there is a commutative diagram
\[
\begin{tikzcd}
    1 \arrow[r] & (\mathcal{O}/\overline{\mathfrak{q}})^\times /\mathcal{O}^\times \arrow[r, "\varphi"] \arrow[d] & G(\overline{\mathfrak{q}}) \arrow[r] \arrow[d] & \operatorname{Cl}_K  \arrow[r] \arrow[d] & 1 \\
    1 \arrow[r] & (\mathcal{O}/\overline{\mathfrak{q}})^\times /(\mathcal{O}/\overline{\mathfrak{q}})^{\times m} \arrow[r] & G(\overline{\mathfrak{q}}) / G(\overline{\mathfrak{q}})^m \arrow[r] & \operatorname{Cl}_K/\operatorname{Cl}_K^m \arrow[r] & 1
\end{tikzcd}
\]
with exact horizontal rows. Here we used that $\mathcal{O}^\times \subseteq (\mathcal{O}/\overline{\mathfrak{q}})^{\times m}$, which follows from our assumption $\gcd(m, w_K) = 1$ as then every element of $\mathcal{O}^\times$ must be an $m^{\operatorname{th}}$ power.

By definition, $\varphi(\kappa) = \mathfrak{q}$. Observe that $p$ certainly splits in $H_K^{(m)}$, as $H_K^{(m)}$ is contained in the ring class field of $\mathcal{O}$. Therefore we conclude that 
\begin{align*}
f(\kappa, K^{(m)}(\overline{\mathfrak{q}})/K) = 1 &\Longleftrightarrow f(\kappa, K^{(m)}(\overline{\mathfrak{q}}) H_K^{(m)}/K) = 1 \\
&\Longleftrightarrow \kappa \text{ is an } m^{\operatorname{th}} \text{ power modulo } \overline{\mathfrak{q}}.
\end{align*}
Now, $\kappa$ is an $m^{\operatorname{th}}$ power modulo $\overline{\mathfrak{q}}$ if and only if $\kappa$ is an $m^{\operatorname{th}}$ power modulo any prime in $\mathcal{O}_M$ above $\overline{\mathfrak{q}}$ if and only if
$$
\left(
\frac{\kappa}{\sigma(\mathfrak{p})}
\right)_{M, m} 
= 1.
$$
Since $\Q(\zeta_m)$ is the fixed field of $\sigma$, we have the equalities
$$
1 = \left(
\frac{\kappa}{\sigma(\mathfrak{p})}
\right)_{M, m} 
\Longleftrightarrow
1 =
\sigma \left(\frac{\kappa}{\sigma(\mathfrak{p})}\right)_{M, m}
= \left(\frac{\sigma(\kappa)}{\sigma^2(\mathfrak{p})}\right)_{M, m}
= \left(\frac{\sigma(\kappa)}{\mathfrak{p}}\right)_{M, m}
= \left(\frac{\overline{\kappa}}{\mathfrak{p}}\right)_{M, m}
= [\mathfrak{p}],
$$
as desired.
\end{proof}

\section{Analytic prerequisites}
In this section we set up the analytic machinery that we will use in the next two sections. Our approach is based on \cite{KM1}. 

\subsection{Cubic residue symbols}
We shall briefly recall the required theory of cubic residue symbols that we need. Let $K$ be a number field always containing a fixed primitive third root of unity that we denote $\zeta_3$. For $\alpha \in O_K$ and $\mathfrak{p}$ a prime of $K$ not dividing $3$, we define $\left(\frac{\alpha}{\mathfrak{p}}\right)_{K, 3}$ as the unique element in $\{1, \zeta_3, \zeta_3^2, 0\}$ satisfying
\[
\left(\frac{\alpha}{\mathfrak{p}}\right)_{K, 3} \equiv \alpha^{\frac{N_{K/\Q}(\mathfrak{p}) - 1}{3}} \bmod \mathfrak{p}.
\]
We multiplicatively extend this to all ideals $I$ coprime to $3$. We shall need the following weak form of cubic reciprocity.

\begin{proposition}
\label{pReciprocity}
Let $M$ be a number field containing $\zeta_3$. Take elements $\alpha, \beta \in O_M$ and assume that $\beta$ is coprime to $3$. Then the cubic residue symbol $(\alpha/\beta)_{M, 3}$ depends only on the congruence class of $\beta$ modulo $27 \alpha O_M$. Moreover, if $\alpha$ is also coprime to $3$, we have
\[
\left(\frac{\alpha}{\beta}\right)_{M, 3} = \mu \cdot \left(\frac{\beta}{\alpha}\right)_{M, 3},
\]
where $\mu$ depends only on the congruence classes of $\alpha$ and $\beta$ modulo $27O_M$.
\end{proposition}

\subsection{Definition of the spin symbol}
Let $K = \Q(\sqrt{-d})$ be an imaginary quadratic field with $d > 0$ squarefree. We will assume that $d \neq 3$. The field $M := K(\zeta_3)$ will play an important role throughout the paper. We fix the following data associated to $M$
\begin{itemize}
\item we let $\{\eta_1, \dots, \eta_4\}$ be an integral basis for $O_M$ with $\eta_1 = 1$,
\item the torsion subgroup of $O_M^\ast$ is generated by $\zeta_{12}$ if $d = 1$, and otherwise the torsion subgroup is generated by $\zeta_6$,
\item we pick two disjoint collections of prime ideals $\mathfrak{p}_1, \dots, \mathfrak{p}_h$ and $\mathfrak{q}_1, \dots, \mathfrak{q}_h$ coprime to $3$, where each collection is a set of representatives for the class group $\CL(M)$ of $M$. We let 
\begin{align}
\label{edeff}
\mathfrak{f} := \prod_{i = 1}^h \mathfrak{p}_i \prod_{j = 1}^h \mathfrak{q}_j, \quad \quad f := N_{M/\Q}(\mathfrak{f}).
\end{align}
Then $\mathfrak{f}$ is principal. Furthermore, we may choose $\mathfrak{p}_1, \dots, \mathfrak{p}_h, \mathfrak{q}_1, \dots, \mathfrak{q}_h$ in such a way that the norm $f$ of $\mathfrak{f}$ is squarefree.
\end{itemize}
The field $M$ is a Galois extension of $\Q$ with Galois group isomorphic to the Klein four group, say $\{1, \sigma, \tau, \sigma \tau\}$, which we depict diagrammatically as
\[
\begin{tikzcd}
& M = K(\zeta_3) & \\
\Q(\zeta_3) \arrow[ur, dash, "\langle \sigma \rangle"] & K = \Q(\sqrt{-d}) \arrow[u, dash, "\langle \tau \rangle"] & \Q(\sqrt{3d}) \arrow[swap, ul, dash, "\langle \sigma \tau \rangle"]\\
& \Q \arrow[ul, dash] \arrow[ur, dash] \arrow[u, dash] & 
\end{tikzcd}
.
\] 
Let $\mathfrak{p}$ be a completely split prime ideal of $M$. If $\mathfrak{p} \tau(\mathfrak{p})$ is a principal ideal of $K$, we define a symbol
\[
[\mathfrak{p}] := \left(\frac{\sigma(\pi)}{\mathfrak{p}}\right)_{M, 3} \in \{1, \zeta_3, \zeta_3^2\}, \quad \quad (\pi)O_K = \mathfrak{p} \tau(\mathfrak{p}) O_K.
\]
This does not depend on the choice of generator $\pi$ of the $O_K$-ideal $\mathfrak{p} \tau(\mathfrak{p}) O_K$, as changing $\pi$ by a unit changes the total symbol by $\left(u/\mathfrak{p}\right)_{M, 3} = 1$ with $u \in \{\pm 1, \pm i\}$ for $d = 1$ and $u \in \{\pm 1\}$ for $d > 1$. One directly checks that
$$
\sum_{\rho \in \Gal(M/\Q)} [\rho(\mathfrak{p})]
=
\begin{cases}
-2 &\textup{if } [\mathfrak{p}] \in \{\zeta_3, \zeta_3^2\} \\
4 &\textup{if } [\mathfrak{p}] = 1.
\end{cases}
$$
More generally, if $\mathfrak{a}$ is any integral ideal of $O_M$ such that $\mathfrak{a} \tau(\mathfrak{a})$ is a principal ideal of $K$, we define
\[
[\mathfrak{a}] := 
\begin{cases}
\left(\frac{\sigma(\alpha)}{\mathfrak{a}}\right)_{M, 3} &\text{if } \gcd(\mathfrak{a}, (3)) = 1 \\
0 &\text{otherwise,}
\end{cases}
\]
where $\alpha O_K$ is any generator of $\mathfrak{a} \tau(\mathfrak{a}) O_K$. This is once more independent of the choice of generator $\alpha$ of the ideal $\mathfrak{a}$.

\subsection{Field lowering}
\label{ssLowering}
The coming results are variations of the results in \cite[Subsection 3.2]{KM1}. As these results play a key role in making our results unconditional, we shall provide full details for these variations. As the results in this subsection hold in significant generality, we allow $K$ to be an arbitrary number field in this subsection only, and we shall return to the setting $K = \Q(\sqrt{-d})$ afterwards.

\begin{lemma}
\label{lFL1}
Let $K$ be a number field and let $\mathfrak{p}$ be a prime of $K$ coprime to $3$. Assume that $L$ is a quadratic extension of $K$ such that $L$ contains $\zeta_3$ and $\mathfrak{p}$ splits in $L$. Write $\sigma$ for the non-trivial element of $\Gal(L/K)$. Then we have for all $\alpha \in O_K$ 
\[
\left(\frac{\alpha}{\mathfrak{p}O_L}\right)_{L, 3} = 
\begin{cases}
\left(\frac{\alpha}{\mathfrak{p}O_K}\right)_{K, 3}^2 &\textup{if } \sigma \textup{ fixes } \zeta_3 \\
\mathbf{1}_{\mathfrak{p} \nmid \alpha} &\textup{if } \sigma \textup{ does not fix } \zeta_3.
\end{cases}
\]
\end{lemma}

\begin{proof}
Since $\mathfrak{p}$ splits in $L$, we may write $\mathfrak{p} O_L = \mathfrak{q} \sigma(\mathfrak{q})$ for some prime $\mathfrak{q}$ of $L$. This gives
\[
\left(\frac{\alpha}{\mathfrak{p}O_L}\right)_{L, 3} = \left(\frac{\alpha}{\mathfrak{q}}\right)_{L, 3} \left(\frac{\alpha}{\sigma(\mathfrak{q})}\right)_{L, 3} = \left(\frac{\alpha}{\mathfrak{q}}\right)_{L, 3} \left(\frac{\sigma(\alpha)}{\sigma(\mathfrak{q})}\right)_{L, 3} = \left(\frac{\alpha}{\mathfrak{q}}\right)_{L, 3} \sigma\left(\left(\frac{\alpha}{\mathfrak{q}}\right)_{L, 3}\right),
\]
where we used that $\sigma(\alpha) = \alpha$ thanks to the assumption $\alpha \in O_K$. If $\sigma$ fixes $\zeta_3$, then the above becomes
\[
\left(\frac{\alpha}{\mathfrak{q}}\right)_{L, 3} \sigma\left(\left(\frac{\alpha}{\mathfrak{q}}\right)_{L, 3}\right) = \left(\frac{\alpha}{\mathfrak{q}}\right)_{L, 3}^2= \left(\frac{\alpha}{\mathfrak{p}O_K}\right)_{K, 3}^2.
\]
If $\sigma$ does not fix $\zeta_3$, then we get $\mathbf{1}_{\mathfrak{p} \nmid \alpha}$ by checking all possibilities for the cubic residue symbol $(\alpha/\mathfrak{q})_{L, 3} \in \{1, \zeta_3, \zeta_3^2, 0\}$.
\end{proof}

\begin{lemma}
\label{lFL2}
Let $K$ be a number field and let $\mathfrak{p}$ be a prime of $K$ coprime to $3$. Assume that $L$ is a quadratic extension of $K$ such that $L$ contains $\zeta_3$ and assume that $\mathfrak{p}$ stays inert in $L$. Further assume that $\mathfrak{p}$ has degree $1$ in $K$ and let $p$ be the prime of $\Q$ lying below $\mathfrak{p}$. Then we have for all $\alpha \in O_K$ 
\[
\left(\frac{\alpha}{\mathfrak{p}O_L}\right)_{L, 3} = 
\begin{cases}
\left(\frac{\alpha}{\mathfrak{p}O_K}\right)_{K, 3}^2 &\textup{if } p \equiv 1 \bmod 3 \\
\mathbf{1}_{\mathfrak{p} \nmid \alpha} &\textup{if } p \equiv 2 \bmod 3. 
\end{cases}
\]
\end{lemma}

\begin{proof}
We have the identities
\[
\left(\frac{\alpha}{\mathfrak{p}O_L}\right)_{L, 3} \equiv \alpha^{\frac{N_{L/\Q}(\mathfrak{p}) - 1}{3}} \equiv \alpha^{\frac{p^2 - 1}{3}}.
\]
If $p \equiv 1 \bmod 3$, we rewrite this as
$$
\alpha^{\frac{p^2 - 1}{3}} \equiv \left(\alpha^{\frac{p - 1}{3}}\right)^{p + 1} \equiv \left(\alpha^{\frac{N_{K/\Q}(\mathfrak{p}) - 1}{3}}\right)^{p + 1} \equiv \left(\frac{\alpha}{\mathfrak{p}O_K}\right)_{K, 3}^{p + 1} \equiv \left(\frac{\alpha}{\mathfrak{p}O_K}\right)_{K, 3}^2 \bmod \mathfrak{p}.
$$
If $p \equiv 2 \bmod 3$, we instead rewrite this as
$$
\alpha^{\frac{p^2 - 1}{3}} \equiv \left(\alpha^{p - 1}\right)^{\frac{p + 1}{3}} \equiv \left(\alpha^{N_{K/\Q}(\mathfrak{p}) - 1}\right)^{\frac{p + 1}{3}} \equiv \mathbf{1}_{\mathfrak{p} \nmid \alpha}^{\frac{p + 1}{3}} \equiv \mathbf{1}_{\mathfrak{p} \nmid \alpha} \bmod \mathfrak{p}, 
$$
as desired.
\end{proof}

\begin{lemma}
\label{lFL3}
Let $K$ be a number field and let $L$ be a quadratic extension of $K$. Write $\sigma$ for the non-trivial element of $\Gal(L/K)$. Assume  that $\mathfrak{p}$ is a prime ideal of $K$ that stays unramified in $L$. Further suppose that $\beta \in O_L$ satisfies $\beta \equiv \sigma(\beta) \bmod \mathfrak{p} O_L$. Then there is $\beta' \in O_K$ such that $\beta' \equiv \beta \bmod \mathfrak{p} O_L$.
\end{lemma}

\begin{proof}
This is proven in \cite[Lemma 3.4]{KM1}.
\end{proof}

\subsection{A fundamental domain}
Given an integral ideal $\mathfrak{a}$ of $M$, there are many possible generators $\alpha$ of $\mathfrak{a}$. For our analytic purposes, it will be useful to pick these generators in a systematic way. To this end, we decompose $O_M^\ast$ as $T \times \langle \epsilon \rangle$, where $T$ is the torsion subgroup and where $\epsilon$ is a fundamental unit. We write $V$ for the group generated by $\epsilon$.

Using our fixed integral basis $\{\eta_1, \eta_2, \eta_3, \eta_4\}$ of $M$, we get an isomorphism, in the category of vector spaces, $\iota: \Q^4 \rightarrow M$ by sending $(a_1, \dots, a_4)$ to $a_1 \eta_1 + a_2 \eta_2 + a_3 \eta_3 + a_4 \eta_4$. We also view $\Q^4$ as a subset of $\R^4$ in the natural way. For a subset $S$ of $\R^4$ and an element $\alpha  = a_1 \eta_1 + a_2 \eta_2 + a_3 \eta_3 + a_4 \eta_4 \in M$, we will abuse notation by writing $\alpha \in S$ to mean $(a_1, \dots, a_4) \in S$. Given a subset $S$ of $\R^4$ and a real number $X > 0$, we define $S(X)$ to be the subset of $(s_1, \dots, s_4) \in S$ satisfying $f(s_1, \dots, s_4) \leq X$, where $f$ is the norm polynomial $f(X_1, \dots, X_4) = N_{M/\Q}(\eta_1 X_1 + \eta_2 X_2 + \eta_3 X_3 + \eta_4 X_4) \in \Z[X_1, \dots, X_4]$. We write $\mathcal{D}$ for the fundamental domain constructed in \cite[Section 3.3]{KM1}. The main properties of this fundamental domain can be found in \cite[Lemma 3.5]{KM1} that we summarize now.

\begin{lemma}
\label{lDomain}
There exists a subset $\mathcal{D} \subseteq \R^4$ with the following properties
\begin{enumerate}
\item[(1)] for all non-zero $\alpha \in O_M$, there exists a unique unit $v \in V$ with $v \alpha \in \mathcal{D}$. Moreover, we have $\{u \in O_M^\ast : u\alpha \in \mathcal{D}\} = \{v t : t \in T\}$;
\item[(2)] the set $\mathcal{D}(1)$ has a $3$-Lipschitz parametrizable boundary. Furthermore, there exists a constant $C > 0$ such that the boundary of $\mathcal{D}(1)$ intersects every line in at most $C$ points;
\item[(3)] there exists a constant $C > 0$ such that for all $\alpha = a_1 \eta_1 + a_2 \eta_2 + a_3 \eta_3 + a_4 \eta_4 \in \mathcal{D}$ (with $a_i \in \Z$), we have $|a_i| \leq C |N_{M/\Q}(\alpha)|^{1/4}$.
\end{enumerate}
\end{lemma}

\subsection{Vinogradov's sieve}
We will now state the main sieve that we will use. This sieve was originally developed by Vinogradov in the context of representing odd integers as sums of three primes. Recall that $M := K(\zeta_3)$. From now on we shall abbreviate $N_{M/\Q}$ simply as $N$ and our ideals will be ideals of $O_M$ unless otherwise indicated.

In order to prove our later results, we will handle a slightly more general symbol than $[\mathfrak{p}]$. To this end, fix an ideal $\mathfrak{M}$ of $M$ and an invertible element $\mu \in O_M/\mathfrak{M}O_M$. We introduce the function $r_i(\mathfrak{a}, \mu, \mathfrak{M})$ for $i = 1, \dots, h$, which is the indicator function of $\mathfrak{a} \mathfrak{p}_i$ having a generator $\alpha$ satisfying $\alpha \equiv \mu \bmod \mathfrak{M}$; recall that $\mathfrak{p}_i$ is one of our fixed representatives for $\CL(M)$. We shall treat $\mu$ and $\mathfrak{M}$ as fixed throughout the paper, and we shall abbreviate $r_i(\mathfrak{a}) := r_i(\mathfrak{a}, \mu, \mathfrak{M})$.

\begin{proposition}
\label{pTypeI} 
For every $\epsilon > 0$, there exists a constant $C > 0$ such that
\[
\left|\sum_{\substack{N(\mathfrak{a}) \leq X \\ \mathfrak{m} \mid \mathfrak{a}}} r_i(\mathfrak{a}) [\mathfrak{a}]\right| \leq C X^{1 - \frac{1}{24} + \epsilon}
\]
uniformly for all non-zero integral ideals $\mathfrak{m}$ of $O_M$ and all $X \geq 2$.
\end{proposition}

\begin{proof}
The proof of this proposition is postponed until Section \ref{sTypeI}.
\end{proof}

\begin{proposition}
\label{pTypeII} 
For every $\epsilon > 0$, there exists a constant $C > 0$ such that
\[
\left|\sum_{N(\mathfrak{m}) \leq M} \sum_{N(\mathfrak{n}) \leq N} \alpha_\mathfrak{m} \beta_\mathfrak{n} r_i(\mathfrak{m} \mathfrak{n}) [\mathfrak{m} \mathfrak{n}]\right| \leq C (M + N)^{\frac{1}{24}} (MN)^{1 - \frac{1}{24} + \epsilon}
\]
uniformly for all $M, N \geq 2$ and all sequences of complex numbers $|\alpha_\mathfrak{m}| \leq 1$ and $|\beta_\mathfrak{n}| \leq 1$.
\end{proposition}

\begin{proof}
The proof of this proposition is postponed until Section \ref{sTypeII}.
\end{proof}

Once Propositions \ref{pTypeI} and \ref{pTypeII} are proven, Vinogradov's sieve for number fields gives the following result.

\begin{theorem}
\label{tMainAnalytic}
For every $\epsilon > 0$, there exists a constant $C > 0$ such that
\[
\left|\sum_{N(\mathfrak{p}) \leq X} r_i(\mathfrak{p}) [\mathfrak{p}]\right| \leq C X^{1 - \frac{1}{3136} + \epsilon}
\]
for all $X \geq 2$.
\end{theorem}

\begin{proof}
Propositions \ref{pTypeI} and \ref{pTypeII} together with Vinogradov's sieve for number fields as presented in \cite[Proposition 5.2]{FIMR} yield
\[
\left|\sum_{N(\mathfrak{a}) \leq X} r_i(\mathfrak{a}) [\mathfrak{a}] \Lambda(\mathfrak{a})\right| \leq C X^{1 - \frac{1}{3136} + \epsilon},
\]
where $\Lambda$ is the natural generalization of the von Mangoldt function to number fields. The theorem is then a direct consequence of partial summation.
\end{proof}

\section{Sums of type I}
\label{sTypeI}
This section is entirely devoted to the proof of Proposition \ref{pTypeI}. Our proof follows among the same lines as \cite[Section 4]{KM1} and \cite[Section 6]{FIMR}. From now on we fix a non-zero integral ideal $\mathfrak{m}$ of $O_M$. Let $\mathfrak{M}$ be an ideal of $O_M$ and let $\mu$ be an element of $O_M$. Recall that we introduced the function $r_i(\mathfrak{a})$ for $i = 1, \dots, h$, which is the indicator function of $\mathfrak{a} \mathfrak{p}_i$ having a generator $\alpha$ satisfying $\alpha \equiv \mu \bmod \mathfrak{M}$ (recall that $\mathfrak{p}_i$ is one of our fixed representatives for $\CL(M)$). We set
$$
F = 2^4 \cdot 3^4 \cdot f \cdot N(\mathfrak{M}) \cdot \Delta(M/\Q), 
$$ 
where $f$ is defined in equation \eqref{edeff} and where $\Delta(M/\Q)$ is the absolute discriminant. Each non-zero ideal has exactly $6$ or $12$ generators in $\mathcal{D}$, which is exactly the size of the torsion subgroup $T$. Further splitting the sum into progressions modulo $F$, we unwrap the sum as 
\begin{align*}
\sum_{\substack{N(\mathfrak{a}) \leq X \\ \mathfrak{m} \mid \mathfrak{a}}} r_i(\mathfrak{a}) [\mathfrak{a}] &= \frac{1}{|T|} \sum_{\substack{\alpha \in \mathcal{D}(X N(\mathfrak{p}_i)) \\ \alpha \equiv 0 \bmod \mathfrak{m} \mathfrak{p}_i}} r_i((\alpha)/\mathfrak{p}_i) [(\alpha)/\mathfrak{p}_i] \\
&= \frac{1}{|T|} \sum_{\rho \bmod F} \sum_{\substack{\alpha \in \mathcal{D}(X N(\mathfrak{p}_i)) \\ \alpha \equiv 0 \bmod \mathfrak{m} \mathfrak{p}_i \\ \alpha \equiv \rho \bmod F}} r_i((\alpha)/\mathfrak{p}_i) [(\alpha)/\mathfrak{p}_i],
\end{align*}
where our first identity is justified due to the change of variables $\mathfrak{a} \mathfrak{p}_i = (\alpha)$. At this point we insert the identity
\begin{align}
\label{eNonprincipal}
[(\alpha)/\mathfrak{p}_i] = [(\alpha)] \cdot \gamma_\rho
\end{align}
for all $\alpha \equiv \rho \bmod F$, where $\gamma_\rho$ is some fixed number in $\{1, \zeta_3, \zeta_3^2, 0\}$ depending only on $\rho$ and $i$. We also note that $r_i((\alpha)/\mathfrak{p}_i)$ depends only on $\alpha \bmod F$. Prompted by these calculations, we define for each $\rho \bmod F$ and each ideal $\mathfrak{m}$
\[
A(X, \rho) = \sum_{\substack{\alpha \in \mathcal{D}(X) \\ \alpha \equiv 0 \bmod \mathfrak{m} \\ \alpha \equiv \rho \bmod F}} [(\alpha)].
\]
By the above manipulations it suffices to estimate each $A(X, \rho)$ individually. Furthermore, we may assume that $\gcd(3, \rho) = (1)$ as otherwise the symbol $[(\alpha)]$ is constantly zero.

Recalling that $\eta_1 = 1$ by convention and writing $\mathbb{M} = \Z \eta_2 + \Z \eta_3 + \Z \eta_4$ allows us to decompose $O_M = \Z \oplus \mathbb{M}$. Therefore every $\alpha \in O_M$ can uniquely be written as
\[
\alpha = a + \beta, \quad \quad a \in \Z, \beta \in \mathbb{M}.
\]
This transforms the sum $A(X, \rho)$ as
\[
A(X, \rho) = \sum_{\substack{a + \beta \in \mathcal{D}(X) \\ a + \beta \equiv 0 \bmod \mathfrak{m} \\ a + \beta \equiv \rho \bmod F}} \left(\frac{\sigma(a + \beta)}{a + \beta}\right)_{M, 3} \left(\frac{\sigma \tau(a + \beta)}{a + \beta}\right)_{M, 3}.
\]
Our aim is now to rewrite the cubic residue symbols. We have
\[
\left(\frac{\sigma(a + \beta)}{a + \beta}\right)_{M, 3} = \left(\frac{a + \sigma(\beta)}{a + \beta}\right)_{M, 3} = \left(\frac{\sigma(\beta) - \beta}{a + \beta}\right)_{M, 3},
\]
and similarly for the other cubic residue symbol. If $\sigma(\beta) - \beta = 0$ or if $\sigma \tau(\beta) - \beta = 0$, then our cubic residue symbols are constantly $0$ and we may safely remove such $\beta$ from consideration. In all other cases, we are thus incentivized to factor $\sigma(\beta) - \beta$ and $\sigma \tau(\beta) - \beta$. Define $\mathfrak{c}$ to be the largest divisor of $\sigma(\beta) - \beta$ coprime to $F$ and define $\mathfrak{c}'$ to be the largest divisor of $\sigma \tau(\beta) - \beta$ coprime to $F$. After applying cubic reciprocity as presented in Proposition \ref{pReciprocity} and following the argument on \cite[p. 726]{FIMR}, we obtain 
\[
\left(\frac{\sigma(\beta) - \beta}{a + \beta}\right)_{M, 3} = \mu_1 \cdot \left(\frac{a + \beta}{\mathfrak{c}}\right)_{M, 3}, \quad \quad \left(\frac{\sigma \tau(\beta) - \beta}{a + \beta}\right)_{M, 3} = \mu_2 \cdot \left(\frac{a + \beta}{\mathfrak{c}'}\right)_{M ,3}
\]
for some numbers $\mu_1, \mu_2 \in \{1, \zeta_3, \zeta_3^2, 0\}$ depending on $\rho$ and $\beta$ but not on $a$. We record our progress by applying the triangle inequality to deduce that
\[
|A(X, \rho)| \leq \sum_{\beta \in \mathbb{M}} |T(X, \rho, \beta)|,
\]
where
\[
T(X, \rho, \beta) := \sum_{\substack{a \in \Z \\ a + \beta \in \mathcal{D}(X) \\ a + \beta \equiv 0 \bmod \mathfrak{m} \\ a + \beta \equiv \rho \bmod F}} \left(\frac{a + \beta}{\mathfrak{c}}\right)_{M, 3} \left(\frac{a + \beta}{\mathfrak{c}'}\right)_{M ,3}.
\]
We will treat $\beta$ as fixed and work towards estimating each $T(X, \rho, \beta)$ individually. It is at this stage that we bring the material from Subsection \ref{ssLowering} into play. To this end, we remark that $\mathfrak{c}'$ is in fact the extension of an ideal of $\Q(\sqrt{3d})$. Here we use that $\mathfrak{c}'$ is coprime to $F$ and thus the discriminant of $M$ and that $\mathfrak{c}'$ divides an element of the shape $\sigma \tau(\beta) - \beta$. From now on we shall view $\mathfrak{c}'$ as an ideal of $O_{\Q(\sqrt{3d})}$. We factor it as
\[
\mathfrak{c}' O_{\Q(\sqrt{3d})} = \prod_{i = 1}^k \mathfrak{p}_i^{e_i}
\]
for prime ideals $\mathfrak{p}_i$ of $O_{\Q(\sqrt{3d})}$, so that 
\[
\left(\frac{a + \beta}{\mathfrak{c}' O_M}\right)_{M ,3} = \prod_{i = 1}^k \left(\frac{a + \beta}{\mathfrak{p}_iO_M}\right)_{M ,3}^{e_i}. 
\]
Observe that the ideals $\mathfrak{p}_i$ are coprime to $F$ by construction and hence unramified in the extension $M/\Q(\sqrt{3d})$. We now claim that
\begin{align}
\label{eSymbolVanishes}
\left(\frac{a + \beta}{\mathfrak{p}_iO_M}\right)_{M ,3} = \mathbf{1}_{\gcd(a + \beta, \mathfrak{p}_i) = (1)}.
\end{align}
Since $\mathfrak{p}_i$ divides $\mathfrak{c}'$ and since $\mathfrak{c}'$ divides $\sigma \tau(\beta) - \beta$, it follows that $\sigma \tau(a + \beta) \equiv a + \beta \bmod \mathfrak{p}_i O_M$. Therefore Lemma \ref{lFL3} allows us to replace $a + \beta$ by some $\beta' \in O_{\Q(\sqrt{3d})}$. Then Lemma \ref{lFL1} gives the claim in case $\mathfrak{p}_i$ splits in $M$. Instead suppose that $\mathfrak{p}_i$ stays inert in $M$. If we write $p_i$ for the unique prime of $\Z$ below $\mathfrak{p}_i$, we certainly find that $p \equiv 2 \bmod 3$. In this case the claim is a consequence of Lemma \ref{lFL2}.

Having proved our claim, we perform a similar operation on the other cubic residue symbol. This residue symbol shall not disappear, but we will be able to ``lower it'' to $\Q(\zeta_3)$. This is the crucial step in making our results unconditional compared to \cite{FIMR} despite working in a field of degree $4$. 

Arguing as before, we may view $\mathfrak{c}$ as the extension of an ideal of $\Z[\zeta_3]$, and by abuse of notation we shall view $\mathfrak{c}$ as an ideal of $\Z[\zeta_3]$ from now on. In this case we factor $\mathfrak{c} \Z[\zeta_3] = \mathfrak{g} \mathfrak{q}$, where $\mathfrak{q}$ has squarefree norm, $\mathfrak{g}$ has squarefull norm and $\gcd(N_{\Q(\zeta_3)/\Q}(\mathfrak{g}), N_{\Q(\zeta_3)/\Q}(\mathfrak{q})) = 1$. Furthermore, observe that $\mathfrak{q}$ and $\mathfrak{g}$ are both coprime to $F$, so all prime divisors of $\mathfrak{q}$ and $\mathfrak{g}$ stay unramified in $M/\Q(\zeta_3)$. So far we have shown that
\[
\left(\frac{a + \beta}{\mathfrak{c} O_M}\right)_{M ,3} = \left(\frac{a + \beta}{\mathfrak{g} O_M}\right)_{M ,3} \left(\frac{a + \beta}{\mathfrak{q} O_M}\right)_{M ,3}.
\]
Lemma \ref{lFL3} and the Chinese remainder theorem allow us to replace $\beta$ by some $\beta' \in \Z[\zeta_3]$. Combining Lemma \ref{lFL1} and Lemma \ref{lFL2} yields
\begin{align}
\label{eLower}
\left(\frac{a + \beta'}{\mathfrak{q} O_M}\right)_{M ,3} = \left(\frac{a + \beta'}{\mathfrak{q}}\right)_{\Z[\zeta_3] ,3}^2,
\end{align}
as any prime $\mathfrak{p}$ in $\Z[\zeta_3]$ that stays inert in $M$ satisfies $p \equiv 1 \bmod 3$ (writing $p := \mathfrak{p} \cap \Z$). Finally, as $q := N_{\Q(\zeta_3)/\Q}(\mathfrak{q})$ is squarefree, the Chinese remainder theorem guarantees the existence of some rational integer $b$ with $\beta' \equiv b \bmod \mathfrak{q}$.

Thanks to equations (\ref{eSymbolVanishes}) and (\ref{eLower}), we have thus arrived at
\[
T(X, \rho, \beta) = \sum_{\substack{a \in \Z \\ a + \beta \in \mathcal{D}(X) \\ a + \beta \equiv 0 \bmod \mathfrak{m} \\ a + \beta \equiv \rho \bmod F}} \left(\frac{a + \beta}{\mathfrak{g} O_M}\right)_{M ,3} \cdot \left(\frac{a + b}{\mathfrak{q}}\right)_{\Z[\zeta_3] ,3}^2 \cdot \mathbf{1}_{\gcd(a + \beta, \mathfrak{c}') = (1)}.
\]
Following the proof on \cite{KM1}, we fix $a$ modulo the radical $g_0$ of $g$ and detect the condition $\mathbf{1}_{\gcd(a + \beta, \mathfrak{c}') = (1)}$ using the Mobius function. Continuing in the footsteps of \cite{KM1}, we apply the triangle inequality
\[
|T(X, \rho, \beta)| \leq \sum_{a_0 \bmod g_0} \sum_{\mathfrak{d} \mid \mathfrak{c}'O_M} |T(X, \rho, \beta, a_0, \mathfrak{d})|,
\]
where 
\[
T(X, \rho, \beta, a_0, \mathfrak{d}) := \sum_{\substack{a \in \Z \\ a + \beta \in \mathcal{D}(X) \\ a + \beta \equiv 0 \bmod \mathfrak{m} \\ a + \beta \equiv \rho \bmod F \\ a \equiv a_0 \bmod g_0 \\ a + \beta \equiv 0 \bmod \mathfrak{d}}} \left(\frac{a + b}{\mathfrak{q}}\right)_{\Z[\zeta_3] ,3}^2.
\]
Since squaring a cubic residue symbol is the same as conjugating it, we may now remove the square factor without changing the absolute value of $T(X, \rho, \beta, a_0, \mathfrak{d})$. After doing so, the sum $T(X, \rho, \beta, a_0, \mathfrak{d})$ has the same shape as \cite[Eq. (4.2)]{KM1}. Having successfully applied our field lowering technique, the remainder of the proof is identical to \cite[p. 7422-7424]{KM1}.

\section{Sums of type II}
\label{sTypeII}
This section is entirely devoted to the proof of Proposition \ref{pTypeII}. This will be relatively straightforward, as the required tools have already been developed in the literature. This was started in \cite[Proposition 3.6]{KM1}, which is unfortunately not applicable here as the symbol is required to take values in $\{\pm 1, 0\}$. Smith \cite{Smi22a} developed a very general large sieve that is completely explicit in its dependence on the underlying number field. Another interesting recent result in this direction can be found in recent work of Santens \cite{Santens}. Here we shall appeal to \cite[Proposition 4.3]{KR}. 

In the notation of \cite{KR}, we shall use \cite[Proposition 4.3]{KR} with the integer $M$ on \cite[p. 11]{KR} equal to $F$, with the number field $K$ equal to our $M = K(\zeta_3)$, with $n$ equal to $4$, with $\ell = 3$, with $t_1 = t_2 = 1$, with the function $\gamma$ equal to
\[
\gamma(w, z) := \left(\frac{\sigma(w)}{z}\right)_{M, 3} \left(\frac{\sigma(z)}{w}\right)_{M, 3} \left(\frac{\sigma \tau(w)}{z}\right)_{M, 3} \left(\frac{\sigma \tau(z)}{w}\right)_{M, 3}
\]
and with $A_{\text{bad}}$ the set of squarefull integers. Then we can take $C_1$ to be an absolute constant ($C_1 = 100$ suffices) and $C_2 = 1/2$. We now check that this data satisfies the conditions (P1)-(P4) on \cite[p. 11-12]{KR}.

Condition (P1) is a consequence of Proposition \ref{pReciprocity}. Condition (P2) follows from the definition of the cubic residue symbol. Condition (P4) is readily verified by counting squarefull integers in the usual way. Finally, (P3) follows from the Chinese remainder theorem and orthogonality of characters. Applying \cite[Proposition 4.3]{KR}, we exhibit, for any given $\epsilon > 0$, the existence of a constant $C(\epsilon) > 0$ satisfying the inequality
\begin{align}
\label{eBilKR1}
\sum_{\substack{w \in \mathcal{D}(X) \\ w \equiv \delta_1 \bmod F}} \sum_{\substack{z \in \mathcal{D}(Y) \\ z \equiv \delta_2 \bmod F}} \alpha_w \beta_z \gamma(w, z) \leq C(\epsilon) (X + Y)^{\frac{1}{24}} (XY)^{1 - \frac{1}{24} + \epsilon}
\end{align}
for all $\alpha_w, \beta_z$ of magnitude bounded by $1$. 

With this calculated, recall that our aim is to demonstrate the bound
$$
\left|\sum_{N(\mathfrak{m}) \leq M} \sum_{N(\mathfrak{n}) \leq N} \alpha_\mathfrak{m} \beta_\mathfrak{n} r_i(\mathfrak{m} \mathfrak{n}) [\mathfrak{m} \mathfrak{n}]\right| \leq C (M + N)^{\frac{1}{24}} (MN)^{1 - \frac{1}{24} + \epsilon}.
$$
Let us first deal with the case that $\mathfrak{m}$ and $\mathfrak{n}$ are both principal. Observe that every principal ideal has exactly $|T|$ generators in $\mathcal{D}$ by Lemma \ref{lDomain}, so our sum becomes
\begin{align}
\label{eBilKR2}
\frac{1}{|T|^2} \sum_{w \in \mathcal{D}(M)} \sum_{z \in \mathcal{D}(N)} \alpha_w \beta_z r_i(wz) [wz].
\end{align}
We note that $r_i(wz)$ depends only on $wz \bmod F$ and that
\[
[(wz)] = \gamma(w, z) \left(\frac{\sigma(w)}{w}\right)_{M, 3} \left(\frac{\sigma(z)}{z}\right)_{M, 3} \left(\frac{\sigma \tau(w)}{w}\right)_{M, 3} \left(\frac{\sigma \tau(z)}{z}\right)_{M, 3}.
\]
Absorbing the last four cubic residue symbols in the coefficients $\alpha_w, \beta_z$ and splitting the sum over congruence classes modulo $F$ shows that we can bound the expression \eqref{eBilKR2} by at most $F^8$ sums appearing in the inequality \eqref{eBilKR1}.

This finishes the proof of Proposition \ref{pTypeII} when the ideals $\mathfrak{m}$, $\mathfrak{n}$ of that proposition are restricted to principal ideals. The general case is handled by appealing to \eqref{eNonprincipal} and reducing to the case of principal ideals.

\section{Proof of main theorems} 
\label{proofs}
In this section, we gather our previous results to prove our main theorems.

\subsection{Proof of Theorem \ref{sieve}}
Define $\mathcal{I}$ to be the set of indices $1 \leq i \leq h$ such that our fixed generator $\mathfrak{p}_i$ of $\CL(M)$ has principal norm in $K$. Then Theorem \ref{sieve} is an immediate consequence of Theorem \ref{tMainAnalytic} by taking $\mathfrak{M} = O_M$, $\mu = 1$ and summing over all $i \in \mathcal{I}$.

\subsection{Proof of Corollary \ref{ellipticcurvedensity}}
We take $\mathcal{I}$ as above, and we write $f$ for the conductor of $\mathcal{O}$. We also take $\mathfrak{M} = fO_M$. The first part of Corollary \ref{ellipticcurvedensity} follows from Proposition \ref{spin and a_p} and Theorem \ref{tMainAnalytic} by summing over all $i \in \mathcal{I}$ and summing over all $\mu \in (O_M/fO_M)^\ast$ whose norm lies in the subgroup of $(O_K/fO_K)^\ast$ generated by $O_K^\ast$ and rational integers coprime to $f$.

The second part of Corollary \ref{ellipticcurvedensity} follows by enlarging $\mathfrak{M}$ and choosing $\mu$ appropriately to account for the additional splitting condition.

\subsection{Proof of Theorem \ref{K^p}}
The first part of Theorem \ref{K^p} is an entirely algebraic statement which is proven by combining Lemma \ref{lSplit1} and Lemma \ref{lSplit2}. The last part of Theorem \ref{K^p} is a consequence of Theorem \ref{tMainAnalytic}.

\subsection{Proof of Corollary \ref{galois}}
We shall now give the necessary background in Galois representations to relate Corollary \ref{galois} to Corollary \ref{ellipticcurvedensity}.

\begin{proof}[Proof of Corollary \ref{galois}]
For any elliptic curve $E$, the action of $\Gal(\overline{\Q}/\Q)$ on the $\ell^\infty$-torsion points of $E$ gives rise to a Galois representation $\rho$ ramified at finitely many places. 
By \cite[Theorem V.2.3.1]{Silverman1}, at each unramified place $p$, the error term $a_\mathfrak{p}(E)$ coincides with the trace of the image under $\rho$ of any Frobenius lift at $p$. Therefore Corollary \ref{galois} is a consequence of Corollary \ref{ellipticcurvedensity}.
\end{proof}

\bibliographystyle{amsplain}
\bibliography{ref.bib}
\end{document}